\documentclass[12pt, a4paper]{article}
\usepackage[margin=1in]{geometry} 

\usepackage[authoryear,longnamesfirst]{natbib}

\usepackage{graphicx}%
\usepackage{multirow}%
\usepackage{amsmath,amssymb,amsfonts}%
\usepackage{amsthm}%
\usepackage{mathrsfs}%
\usepackage[title]{appendix}%
\usepackage{xcolor}%
\usepackage{textcomp}%
\usepackage{manyfoot}%
\usepackage{booktabs}%
\usepackage{algorithmicx}%
\usepackage{algpseudocode}%
\usepackage{listings}%

\usepackage{subfigure}
\usepackage{empheq}
\usepackage{url}

\usepackage{authblk}

\newtheorem{theorem}{Theorem}
\newtheorem{proposition}[theorem]{Proposition}%
\newtheorem{remark}{Remark}%

\newtheorem{definition}{Definition}%

\newtheorem{lemma}[theorem]{Lemma}

\newtheorem{corollary}{Corollary}

\begin{document}



\title{Eventual convexity for separable chance constraints with skewed generalized hyperbolic random variables}  

\author[1]{Heng Zhang\thanks{Corresponding author: hengzhang@centralesupelec.fr}}
\author[1]{Abdel Lisser}

\affil[1]{Universit\'e Paris-Saclay, CNRS, CentraleSup\'elec, Laboratory of signals and systems, 3, Rue Joliot-Curie, 91192 Gif sur Yvette, France}

\date{} 

\maketitle

\begin{abstract}
Chance constraints are widely used in optimization under uncertainty. This paper aims to show the eventual convexity (EV) of chance constraints with skewed generalized hyperbolic (GH) random variables. We prove that the densities of GH distributions are $\alpha$-decreasing, and obtain exact convex reformulations for separable jointly chance constraints with GH distributions. We provide numerical results to compute the $\alpha$-decreasing threshold parameters and show the EV of the feasible set together with the computational tractability to solve the associated optimization problems.

\vspace{1em}
\end{abstract}

\section{Introduction}
We consider the following chance constrained optimization (CCO) problem:
\begin{equation}\label{CCP}
\min c^{\top}x \ \text{ s.t. } \ \mathbb{P}(\xi\le g(x))\ge p, \ x\in X,  
\end{equation}
where $X \subseteq \mathbb{R}^{N}(N \geq 1)$ is a closed convex set, $\xi=(\xi_1,\ldots,\xi_m)^\top \in \mathbb{R}^m$ is an $m$-dimensional random vector, and $g(x)=(g_1(x),\ldots,g_m(x))^\top : X \to \mathbb{R}^m$ is a deterministic vector-valued mapping.

Chance constraints, initially introduced by Charnes and Cooper \cite{charnes1963deterministic} and Charnes et al. \cite{charnes1958cost}, were used to describe the probability of events occurring beyond a given probability value. Prékopa \cite{prekopa1995stochastic, prekopa2003probabilistic} considered the general chance constraint
\begin{equation}\label{Introduction 1}
\mathbb{P}\left( h(x, \xi) \leq 0 \right) \geq p,
\end{equation}
where $p \in [0,1]$ is the confidence level, $\xi\in \mathbb{R}^{m}$ is a random vector, $x \in \mathbb{R}^{n}$ is the decision vector, and $h: \mathbb{R}^{n} \times \mathbb{R}^{m} \to \mathbb{R}^{k}$ is a vector-valued mapping. A classical result of Prékopa \cite[Theorem 10.2.1]{prekopa1995stochastic} stated that if the density of $\xi$ is log-concave and the components of $-h$ are quasi-concave, then the feasible set of \eqref{Introduction 1} is convex for all $p \in [0,1]$. However, the function $-h$ can not always be quasi-concave. A counterexample in \cite{kataoka1963stochastic} showed that when $h(x, \xi)= x^{\top} \xi$ and $\xi$ subject to a multivariate Gaussian distribution, then the probability threshold $p = 1/2$ acts as a sharp boundary for the convexity. This motivates the notion of \emph{eventual convexity}, which refers to convexity of the feasible set for sufficiently large probability levels \cite{henrion2008convexity,henrion2011convexity}. For separable jointly chance constraints,
\begin{equation}\label{Introduction 2}
\mathbb{P}\left( \xi \leq g(x) \right) \geq p,
\end{equation}
Henrion and Strugarek \cite{henrion2008convexity} proposed a useful criterion: if $g$ is $-r$-concave and the distribution function of $\xi$ has an $(r+1)$-decreasing density, then the feasible set of \eqref{Introduction 2} is eventually convex. 
While eventual convexity is well-studied for symmetric elliptical distributions \cite{cheng2014second,minoux2016convexity,nguyen2023convexity,zhang2025convexity}, many financial and actuarial applications exhibit skewness and heavy tails \cite{davies2009fund,lane2000pricing}, often modeled via the normal mean-variance mixture (NMVM) or generalized hyperbolic (GH) distributions \cite{birge2021portfolio, wang2022portfolio}. Recently, chance constraints with skewed distributions have gained attention \cite{ke2015novel,lasserre2021distributionally,nguyen2024random,peng2021chance}. In particular, Liu et al. \cite{liu2022chance} proposed algorithmic approximations for GH chance constraints. However, the theoretical eventual convexity of GH chance constraints remains unexplored. The main difficulty lies in the complicated density structure of skewed distributions, such as the intricate form of Bessel functions embedded within the definition of GH distributions.

In this paper, we investigate the eventual convexity of separable jointly chance constraints with skewed GH random variables. Our main theoretical results rely specifically on the assumption of separable chance constraints. Our first contribution is to establish the $\alpha$-decreasing property of GH densities by deriving suitable estimates for the modified Bessel functions. Based on this result, we obtain eventual convexity for separable jointly chance constraints with GH distributions. We also provide numerical experiments to compute the threshold parameters associated with the $\alpha$-decreasing property and to illustrate the eventual convexity of the feasible set.

The rest of this paper is organized as follows. Section \ref{Preliminaries} collects the necessary preliminaries on skewed distributions. Section \ref{section about alpha decreasing propoerty} proves the $\alpha$-decreasing property of GH distributions. Section \ref{subsection about eventual convexity for separable} establishes the eventual convexity result and the convex reformulations for separable jointly chance constraints with GH distributions. Section \ref{section about numerical experiments} presents numerical experiments on the threshold parameters, the eventual convexity of the feasible sets, and the solutions of the associated optimization problems.

\section{Fundamentals of skewed distributions}\label{Preliminaries}

Normal variance mixtures provide a tractable extension of Gaussian distributions and form the basis of the generalized hyperbolic family. Throughout this section, we assume that $\Sigma$ is positive definite and that the mixing variable $W$ has no atom at zero.


\begin{definition}\label{definition of normal mean variance mixture distribution}
A random vector $\varsigma$ follows a (multivariate) \textbf{NMVM} distribution if
\begin{equation}\label{normal mean variance mixture distribution}
\varsigma \overset{d}{=} m(W)+\sqrt{W}\,AZ,
\end{equation}
where $Z \sim \mathcal{N}_{N}(0,I_{N})$, $W\ge 0$ is a scalar random variable independent of $Z$, $A\in \mathbb{R}^{N\times N}$, and $m:[0,\infty)\to\mathbb{R}^{N}$ is measurable. 

The NMVM model allows the mean vector to depend on $W$, which introduces asymmetry into the random vectors.
\end{definition}

\begin{definition}\label{definition of Generalized inverse Gaussian distributions}
The random variable $Y$ has a \textbf{generalized inverse Gaussian (GIG)} distribution, i.e., $Y\sim Gig(\lambda,\chi,\psi)$, if its density is
$
f_{Y}(t) = \frac{\chi^{-\lambda}\left( \sqrt{ \chi \psi} \right)^{\lambda}}{2 K_{\lambda} \left( \sqrt{\chi \psi} \right)} t^{\lambda - 1} \exp \left( -\frac{1}{2}\left( \chi t^{-1} + \psi t \right) \right)$ for $t > 0$, where $K_{\lambda}$ is the modified Bessel function of the third kind and the parameters satisfy
\begin{subequations}\label{eq:all_conditions}
\begin{empheq}[left=\empheqlbrace]{align}
    \chi > 0, \psi \geq 0, &\quad \text{if } \lambda < 0, \label{eq:cond_a} \\         
    \chi > 0, \psi > 0, &\quad \text{if } \lambda = 0, \label{eq:cond_b} \\
    \chi \geq 0, \psi > 0, &\quad \text{if } \lambda > 0. \label{eq:cond_c}
\end{empheq}
\end{subequations}
In particular, when $\psi = 0$, GIG density contains the \textbf{inverse gamma (IG)} density as a special limiting case. The random variable $Y$ follows an IG distribution, i.e., $Y \sim Ig(\alpha, \beta)$, if its density is $f_Y(t)=\frac{\beta^\alpha}{\Gamma(\alpha)}\,t^{-(\alpha+1)}e^{-\beta/t}$ for $t>0,\ \alpha>0,\ \beta>0$.
Assume $\lambda < 0$ and $\psi = 0$. Then, the random variable $Y \sim Gig(\lambda, \chi, 0)$ is equivalent to $Y \sim Ig(-\lambda, \frac{1}{2} \chi)$, which means the density of $Y$ can be written as 
$
f_Y(t)=\left(\frac{1}{2}\chi\right)^{-\lambda}(\Gamma(-\lambda))^{-1}\,t^{\lambda-1}\exp\!\left(-\frac{1}{2}\chi\,t^{-1}\right).
$
Likewise, when $\chi=0$ and $\lambda>0$, the GIG distribution reduces in the limiting sense to a gamma distribution $Ga(\lambda,\psi/2)$, with density
$
f_Y(t)=\frac{(\psi/2)^\lambda}{\Gamma(\lambda)},t^{\lambda-1}e^{-\psi t/2}$ for $t>0$.
\end{definition}

\begin{definition}\label{definition of Generalized hyperbolic distributions}
A random vector $\varsigma$ follows a (multivariate) \textbf{generalized hyperbolic} (\textbf{GH}) distribution if \eqref{normal mean variance mixture distribution} holds with $W\sim Gig(\lambda,\chi,\psi)$ and $m(W)=\mu+W\gamma$, where $\mu,\gamma\in\mathbb{R}^{N}$. We denote it briefly by $\varsigma \sim GH_{N}(\lambda,\chi,\psi,\mu,\Sigma,\gamma)$.
\end{definition}

The GH family encompasses a rich class of distributions, accommodating both skewness and heavy tails. Notable special cases \cite{jiang2023review,mcneil2005quantitative} parameterized by $(\lambda, \chi, \psi, \gamma)$ include the Hyperbolic $(\lambda=(N+1)/2)$, Normal Inverse Gaussian $(\lambda=-1/2)$, Variance Gamma $(\chi=0)$, and generalized Skewed Student's t $(\psi=0)$. When $\gamma=0$, it reduces to symmetric elliptical distributions (e.g., Gaussian, Student's t, Laplace), which have been extensively studied in previous eventual convexity literature. 

Let $Z \sim GH_{N}(\lambda,\chi,\psi,\mu,\Sigma,\gamma)$, and let $h(\omega)$ denote the density of $W$. Then, the density of $Z$ admits the mixture representation
\begin{equation}\label{density of GH_d random vector with W density form}
\begin{split}
&f_{Z}(t)=\int_{0}^{+\infty}
\frac{\exp\!\left((t-\mu)^{\top}\Sigma^{-1}\gamma\right)}
{(2\pi)^{N/2}|\Sigma|^{1/2}\omega^{N/2}} \cdot\\
&\exp\!\left\{-\frac{(t-\mu)^{\top}\Sigma^{-1}(t-\mu)}{2\omega}
-\frac{\gamma^{\top}\Sigma^{-1}\gamma}{2/\omega}\right\}
h(\omega)\,d\omega.
\end{split}
\end{equation}
Evaluating \eqref{density of GH_d random vector with W density form} yields the closed form
\begin{equation}\label{non-integral form _ density of GH_d random vector with W density form}
f_{Z}(t)=
c\,
\frac{
K_{\lambda-(N/2)}\!\left(S(t)
\right)
\exp\!\left((t-\mu)^{\top}\Sigma^{-1}\gamma\right)
}{
\left(
S(t)
\right)^{(N/2)-\lambda}
},
\end{equation}
where $S(t) := \sqrt{
\left(\chi+(t-\mu)^{\top}\Sigma^{-1}(t-\mu)\right)
\left(\psi+\gamma^{\top}\Sigma^{-1}\gamma\right)
}$,
\[
c=
\frac{
\left(\sqrt{\chi\psi}\right)^{-\lambda}\psi^{\lambda}
\left(\psi+\gamma^{\top}\Sigma^{-1}\gamma\right)^{(N/2)-\lambda}
}{
(2\pi)^{N/2}|\Sigma|^{1/2}K_{\lambda}\!\left(\sqrt{\chi\psi}\right)
}.
\]
It is important to note that when either $\chi = 0$ or $\psi = 0$ (i.e., $\chi\psi = 0$), the explicit formula for $f_Z(t)$ and its normalizing constant $c$ in \eqref{non-integral form _ density of GH_d random vector with W density form} contains $K_\lambda(0)$, which diverges. In such cases, the density equation is strictly understood in a limiting sense as $\chi\psi \to 0^+$. Thanks to the asymptotic approximation of the Bessel function for small arguments (i.e., $K_\lambda(x) \simeq \Gamma(|\lambda|) 2^{|\lambda|-1} x^{-|\lambda|}$ as $x \to 0^+$), this limit is well-defined and yields the specific densities for the subclasses of GH distributions (e.g., Variance-Gamma and skewed Student's t). An explicit derivation of this limiting density for the case $\psi=0$ is provided in Corollary \ref{density of special GH_1}.



\begin{proposition}\cite[Proposition 3.13.]{mcneil2005quantitative}\label{scalarization of generalized hyperbolic distributions}
If $Y \sim GH_{N}(\lambda,\chi,\psi,\mu,\Sigma,\gamma)$ and $Z = BY + b$, where $B \in \mathbb{R}^{k\times N}$ and $b \in \mathbb{R}^{k}$, then $Z \sim GH_{k}(\lambda,\chi,\psi,\; B\mu+b,\; B\Sigma B^{\top},\; B\gamma)$.
\end{proposition}

\begin{proposition}\label{1-dimensional generalized hyperbolic distributions}
Let $x \in \mathbb{R}^{N}\setminus\{0\}$. Suppose that $\xi \sim GH_{N}(\lambda,\chi,\psi,\mu,\Sigma,\gamma)$ and define $\zeta(x):=(\xi^{\top}x-\mu^{\top}x)/\sqrt{x^{\top}\Sigma x}$. Then, $\zeta(x) \sim GH_{1}\!(\lambda,\chi,\psi,0,1,\frac{x^{\top}\gamma}{\sqrt{x^{\top}\Sigma x}})$.
\end{proposition}

\begin{proof}
Set $B:=x^{\top}/\sqrt{x^{\top}\Sigma x}$ and $
b:=-x^{\top}\mu/\sqrt{x^{\top}\Sigma x}$. The result follows immediately from Proposition \ref{scalarization of generalized hyperbolic distributions}.
\end{proof}

The following corollary gives a one-dimensional GH density in the case $\psi=0$. Given that $\psi=0$ is admissible exclusively when $\lambda<0$ and $\chi>0$, we restrict our attention to case \eqref{eq:cond_a}.

\begin{corollary}\label{density of special GH_1}
Assume that $\lambda<0$ and $\chi>0$. If $\xi \sim GH_{1}(\lambda,\chi,0,0,1,0)$, then the density of $\xi$ is given by $f_{\xi}(t)=\hat{c}\,(\chi+t^{2})^{\lambda-\frac12}$, where $\hat{c}:=\frac{\Gamma\!\left(-\lambda+\frac12\right)}{\Gamma(-\lambda)}\frac{1}{\sqrt{\pi}}\chi^{-\lambda}$.
\end{corollary}

\begin{proof}
Since $Gig(\lambda,\chi,0)$ is equivalent to $Ig(-\lambda,\chi/2)$, by Definition \ref{definition of Generalized inverse Gaussian distributions}, we obtain
\begin{equation}\label{limitation of coefficient}
\lim_{s \to 0^{+}} \frac{s^{\lambda}}{K_{\lambda}(s)}=\frac{2^{\lambda+1}}{\Gamma(-\lambda)}.
\end{equation}
The claim follows by taking the limit $\psi\to 0^{+}$ in the one-dimensional GH density \eqref{non-integral form _ density of GH_d random vector with W density form} and using \eqref{limitation of coefficient}.
\end{proof}

\section{$\alpha$-decreasing property of skewed GH distributions}\label{section about alpha decreasing propoerty}
In this section, we derive the $\alpha$-decreasing of GH densities, where the $\alpha$-decreasing property is defined as follows.
\begin{definition}\label{def of r-decreasing}
A function $f : \mathbb{R} \rightarrow \mathbb{R}$ is \textbf{$\alpha$-decreasing} for some $\alpha \in \mathbb{R}$ if $f$ is continuous on $(0, +\infty)$ and there exists some $t^{*}(\alpha) > 0$ such that the mapping $t \mapsto t^{\alpha} f(t)$ is strictly decreasing for all $t > t^{*}(\alpha)$.
\end{definition}
The density of GH distribution is defined by the modified Bessel functions of the third kind, which is a kind of extreme non-linear function \cite{arfken2013mathematical, gil2002evaluation}. Thus, it is more difficult to prove the $\alpha$-decreasing property than the cases with Gaussian, t-Student, and other elliptical distributions \cite[Proposition 4.1]{henrion2008convexity}, \cite[Proposition 5]{cheng2014second}. 
At the beginning, we collect some properties about the modified Bessel functions of the third kind in the following. 
\begin{proposition}\cite{spanier1987atlas, gil2002evaluation}\label{The properties of modified Bessel functions of the third kind}
    Let $x > 0$ and $\nu \in \mathbb{R}$. The modified Bessel
    function of the third kind $K_\nu(x)$ satisfies: (1) $K_{\nu}(x) > 0$; (2) $K_{\nu}(x)$ is decreasing; (3) $\lim_{x \to +\infty} K_{\nu}(x) = 0$; (4) $K_{\nu}(x) = K_{-\nu}(x)$; (5) $K_{\nu}(x)$ increases as $| \nu |$ increasing; (6) For large $x \gg \mu := \frac{\nu^{2}}{2} - \frac{1}{8}$, $K_{\nu}(x) \simeq \sqrt{\frac{\pi}{2x}} \cdot exp(-x) \cdot \left( 1 + \frac{1}{x} \right)^{\mu}$; (7) $K^{'}_{\nu}(x) = \frac{\nu}{x} \cdot K_{\nu}(x) - K_{\nu + 1}(x)$; (8) For $x \to 0^{+}$, $K_{\nu}(x) \simeq \Gamma(\nu) \cdot 2^{\nu - 1} \cdot x^{-\nu}$ if $\nu > 0$, $K_{\nu}(x) \simeq \Gamma(-\nu) \cdot 2^{-\nu - 1} \cdot x^{\nu}$ if $\nu < 0$, and $K_{\nu}(x) \simeq -\ln (x) + \ln 2 - \hat{\gamma}$ if $\nu = 0$ (with Euler's constant $\hat{\gamma} \approx 0.5772$).
\end{proposition}

Suppose $\xi \sim GH_{1}(\lambda,\chi,\psi,0,1,\varphi)$. For the sake of simplicity, we make the following definitions throughout this paper:
\begin{equation}\label{auxiliary values 1}
\eta(t) := \sqrt{\chi+t^{2}}, \ \ \rho(t) := e^{\varphi t}, \ \ \Lambda := \sqrt{\psi+\varphi^{2}}.
\end{equation}
Direct computation yields $\eta'(t) = t/\sqrt{\chi+t^{2}}$ and $ \rho'(t) = \varphi \rho(t)$.

\begin{lemma}\label{preparation result of alpha-decreasing 1}
Assume $\xi \sim GH_{1}(\lambda,\chi,\psi,0,1,\varphi)$ and $\Lambda > 0$. Let $J$ be defined on $(0,+\infty)$ by
\begin{equation}\label{the division about Bessel functions}
J(t):=\frac{K_{\lambda-\frac12}(\Lambda\,\eta(t))}{K_{\lambda+\frac12}(\Lambda\,\eta(t))}.
\end{equation}
Then, there exists $t_{0}>0$ such that
\begin{equation}\label{boundary of J}
    0<J(t)\le \sup_{t>t_{0}}J(t)=:c_{0}(t_{0})<+\infty,\qquad \forall t>t_{0}.
\end{equation}
Moreover, for any $t>0$, the following holds:
\begin{equation}\label{the boundary of c_0}
\left\{
        \begin{array}{ll}
        0 < J(t) < 1, & \text{if} \ \ \lambda > 0,\\            
        J(t) \equiv 1, & \text{if} \ \ \lambda = 0,\\
        J(t) > 1, & \text{if} \ \ \lambda < 0,
        \end{array}
\right.
\end{equation}
and
\begin{equation}\label{limitation of J}
    J(t)\simeq \left(\frac{1}{1+\frac{1}{\Lambda\,\eta(t)}}\right)^{\lambda}\to 1
\qquad \text{as } t\to+\infty.
\end{equation}
\end{lemma}
\begin{proof}
Applying properties (1) and (6) of Proposition \ref{The properties of modified Bessel functions of the third kind}, we obtain
\begin{equation}\label{the property of the division about Bessel functions}
\begin{split}
&0 < \frac{K_{\lambda - (1/2)}(t)}{K_{\lambda + (1/2)}(t)} \simeq  \left( \frac{1}{1 + \frac{1}{t}} \right)^{\lambda} \to 1 \quad \text{as} \ \ t \to +\infty.
\end{split}
\end{equation}
From \eqref{auxiliary values 1}, $\Lambda \eta(t)$ is strictly increasing for $t>0$ with $\lim_{t \to +\infty} \Lambda \eta(t) = +\infty$. Thus, \eqref{boundary of J} and \eqref{limitation of J} follow immediately from \eqref{the property of the division about Bessel functions}. Using properties (4) and (5) of Proposition \ref{The properties of modified Bessel functions of the third kind}, we note that $K_{\lambda - (1/2)} < K_{\lambda + (1/2)}$ if and only if $|\lambda - 1/2| < |\lambda + 1/2|$. Evaluating these absolute values yields \eqref{the boundary of c_0}, which completes the proof.
\end{proof}

\begin{lemma}\label{preparation result of alpha-decreasing 2}
Under the assumptions of Lemma \ref{preparation result of alpha-decreasing 1}, assume $\Lambda > 0$. Then, the function $J$, defined by  \eqref{the division about Bessel functions}, satisfies $\lim_{t \to +\infty} t  \left( 1 - J^{2}(t) \right) = 2\lambda/\Lambda$.
\end{lemma}
\begin{proof}
    If $\lambda = 0$, $J(t) \equiv 1$ trivially by property (4) of Proposition \ref{The properties of modified Bessel functions of the third kind}. For $\lambda \neq 0$, we use
    the two-term asymptotic expansion
    $K_\nu(x) = \sqrt{\frac{\pi}{2x}} e^{-x} ( 1 + \frac{4\nu^2 - 1}{8x} + O(x^{-2}) )$
    as $x \to \infty$ \cite[Page 503]{spanier1987atlas}. Setting $x = \Lambda \eta(t)$ and applying this to
    $K_{\lambda - 1/2}(x)$ and $K_{\lambda + 1/2}(x)$, their ratio yields
    $J(t) = 1 - \frac{\lambda}{x} + O(x^{-2})$. Squaring $J(t)$ gives
    $1 - J^2(t) = \frac{2\lambda}{x} + O(x^{-2})$. Since
    $x = \Lambda\eta(t) \sim \Lambda t$ as $t \to \infty$, it directly follows that
    $\lim_{t\to\infty} t(1 - J^2(t)) = \frac{2\lambda}{\Lambda}$.
\end{proof}

The following theorem provides the main result of this section about the $\alpha$-decreasing property of GH densities. In contrast to the classical elliptical distributions, it is more challenging to discuss the $\alpha$-decreasing property of skewed distributions due to the complexity of the Bessel functions. A set of estimates relevant to the modified Bessel function of the third kind plays a crucial role in proving the following theorem.
\begin{theorem}\label{alpha-decreasing of density function}
Let $\xi \sim GH_{1}(\lambda,\chi,\psi,0,1,\varphi)$. Then, we have the following assertions:
\begin{itemize}
    \item If $\psi>0$, then for any $\alpha\in\mathbb{R}$ there exists $t^{*}(\alpha)>0$ such that the density function $f_{\xi}$ is $\alpha$-decreasing on $(t^{*}(\alpha),+\infty)$.
    \item Suppose $\psi=0$ and $\varphi\neq 0$. If $\varphi<0$, then $f_{\xi}$ is $\alpha$-decreasing for any $\alpha\in\mathbb{R}$. If $\varphi>0$, then $f_{\xi}$ is $\alpha$-decreasing for any $\alpha<1-\lambda$.
    \item Suppose $\psi=0$ and $\varphi=0$. Then $f_{\xi}$ is $\alpha$-decreasing if and only if $\lambda<0$ and $\alpha<1-2\lambda$.
\end{itemize}
\end{theorem}

\begin{proof}
    To prove that $f_{\xi}$ is $\alpha$-decreasing, it suffices to show that there exists $t^{*}(\alpha)>0$ such that \begin{equation}\label{alpha_decreasing proof} t f'_{\xi}(t)+\alpha f_{\xi}(t)<0,\qquad \forall t>t^{*}(\alpha). \end{equation}
    We first consider the case $\Lambda = \sqrt{\psi+\varphi^{2}}>0$. By substituting \eqref{auxiliary values 1} into \eqref{non-integral form _ density of GH_d random vector with W density form} and applying Proposition \ref{The properties of modified Bessel functions of the third kind} (7), the condition $tf_\xi'(t) + \alpha f_\xi(t) < 0$ is equivalent to
    \begin{equation}\label{alpha_decreasing proof-simplification-5}
    \begin{aligned}
    & J(t) \left( 2 \lambda - 1 \right) + J(t) \alpha - J(t) \frac{\chi  \left( 2 \lambda -1 \right)}{\chi + t^{2}} +  \frac{\chi \sqrt{\psi + \varphi^{2}}}{\sqrt{ \chi + t^{2}}}\\
    & < \sqrt{\psi + \varphi^{2}}  \left( \sqrt{\chi + t^{2}} - t \right) + t \left( \sqrt{\psi + \varphi^{2}} - J(t) \varphi \right)\\
    &=: Q(t).
    \end{aligned}
    \end{equation}
    where $J(t)$ is defined in \eqref{the division about Bessel functions}. Using Lemma \ref{preparation result of alpha-decreasing 1}, the left-hand side of \eqref{alpha_decreasing proof-simplification-5} tends to $2 \lambda - 1 + \alpha$ as $t \to +\infty$. 

    Next, we evaluate the right-hand side $Q(t)$ using Lemma \ref{preparation result of alpha-decreasing 2} and the limits established in Lemma \ref{preparation result of alpha-decreasing 1}: 
    \begin{itemize}
        \item If $\psi > 0$, algebraic expansion and limit analysis yield $\lim_{t\to+\infty} Q(t) = +\infty$. Thus, inequality \eqref{alpha_decreasing proof-simplification-5} holds for sufficiently large $t$ for any $\alpha \in \mathbb{R}$.
        \item If $\psi = 0$ and $\varphi > 0$, we have $\lim_{t\to+\infty} Q(t) = \lambda$. For \eqref{alpha_decreasing proof-simplification-5} to eventually hold, we require $2\lambda - 1 + \alpha < \lambda$, which yields $\alpha < 1 - \lambda$.
        \item If $\psi = 0$ and $\varphi < 0$, reformulating $Q(t)$ similarly yields $\lim_{t\to+\infty} Q(t) = +\infty$, implying the condition holds for any $\alpha \in \mathbb{R}$.
    \end{itemize}

    Finally, we consider the case with $\psi=0$ and $\varphi=0$. Using Corollary \ref{density of special GH_1}, the density simplifies to $f_\xi(t) = \hat{c}(\chi + t^2)^{\lambda - 1/2}$. Direct differentiation reduces the condition $tf_\xi'(t) + \alpha f_\xi(t) < 0$ to a simple quadratic inequality: $(2\lambda - 1 + \alpha)t^2 + \alpha\chi < 0$. For this to hold for sufficiently large $t$, the leading coefficient must be strictly negative, which gives $\alpha < 1 - 2\lambda$. This completes the proof.
\end{proof}


\section{Eventual convexity for separable chance constraints with skewed GH distributions}\label{subsection about eventual convexity for separable}
In this section, we investigate the eventual convexity of separable chance constraints subject to generalized hyperbolic distributed perturbations. Consider the separable chance-constrained model
\[
M(p)=\{x\in X:\mathbb{P}(\xi\le g(x))\ge p\},
\]
where $X \subset \mathbb{R}^{N}$ is a non-empty and convex set. In the case of independent components, convexity follows from Theorem 3.1 in \cite{henrion2008convexity} if the marginal densities satisfy the required $\alpha$-decreasing property. Theorem \ref{alpha-decreasing of density function} yields explicit parameter conditions guaranteeing this property.

\begin{definition}
A function $f : \Omega \rightarrow [0, +\infty)$ on a set $\Omega \subset \mathbb{R}^{N}$ is called \textbf{$\alpha$-concave} for some $\alpha \in [-\infty, +\infty]$ if and only if for all $x, y \in \Omega$, $\lambda \in [0, 1]$ the following inequality holds:
\begin{equation*}
f(\lambda x + (1 - \lambda) y) \geq m_{\alpha} ( f(x), f(y), \lambda ),
\end{equation*}
where $m_{\alpha} : [0, +\infty) \times [0, +\infty) \times [0, 1] \rightarrow \mathbb{R}$ is defined as follows:
$
m_{\alpha}(a, b, \lambda) = 0 \ \ \text{if} \ \ ab = 0,$
and if $a > 0$, $b > 0$, $0 \leq \lambda \leq 1$, then
\begin{equation*}
m_{\alpha}(a, b, \lambda) =
\left\{
        \begin{array}{ll}
        a^{\lambda}b^{1-\lambda} & \text{if} \ \ \alpha = 0,\\
        \text{max}\lbrace a, b \rbrace & \text{if} \ \ \alpha = +\infty,\\
        \text{min}\lbrace a, b \rbrace & \text{if} \ \ \alpha = -\infty \\
        (\lambda a^{\alpha} + (1 - \lambda) b^{\alpha})^{1/\alpha} & otherwise.
        \end{array}
\right.
\end{equation*}
\end{definition}

\begin{theorem}\label{thm:eventual_convexity_standardized}
Let $\xi_i \sim GH_1(\lambda_i,\chi_i,\psi_i,0,1,\varphi_i),\  i=1,\dots,m$. Suppose that $\{\xi_i\}_{i=1}^m$ are mutually independent. Let $r_i>0$ $(i = 1, ..., m)$. Assume that, for each $i$, the parameters of $\xi_i$ satisfy one of the following conditions:
\begin{itemize}
    \item $\psi_i>0$;
    \item $\psi_i=0$ and $\varphi_i<0$;
    \item $\psi_i=0$, $\varphi_i>0$, and $r_i<-\lambda_i$;
    \item $\psi_i=0$, $\varphi_i=0$, $\lambda_i<0$, and $r_i<-2\lambda_i$.
\end{itemize}
Suppose that each $g_i$ is $(-r_i)$-concave on a convex set $X \subset \mathbb{R}^{N}$. Then the feasible set $M(p):=\left\{x\in X:\prod_{i=1}^m \mathbb{P}(\xi_i\le g_i(x))\ge p\right\}$ is convex for all $p>\max_{1\le i\le m} F_{\xi_i}(t_i^*)$, where $t_i^*$ denotes the threshold associated with the $(r_i+1)$-decreasing density of $\xi_i$.
\end{theorem}

\begin{proof}
Under the stated parameter conditions, Theorem \ref{alpha-decreasing of density function} implies that each density of $\xi_i$ is $(r_i+1)$-decreasing. Thus, the results follow directly from Theorem 3.1 in \cite{henrion2008convexity}.
\end{proof}

\begin{corollary}\label{cor:eventual_convexity_general}
Let $\xi_i \sim GH_1(\lambda_i,\chi_i,\psi_i,\mu_i,\sigma_i^2,\varphi_i)$ with $\sigma_i>0$ $(i=1,\dots,m)$. Suppose that $\{\xi_i\}_{i=1}^m$ are mutually independent. Define $\zeta_i:=(\xi_i-\mu_i)/\sigma_i$ and $\tilde g_i(x):=(g_i(x)-\mu_i)/\sigma_i$. Assume that each $\tilde g_i$ is $(-r_i)$-concave on a convex set $X \subset \mathbb{R}^{N}$, and that each $\zeta_i$ satisfies the parameter conditions in Theorem \ref{thm:eventual_convexity_standardized} with the same $r_i$. Then, $
M(p):=\left\{x\in X:\prod_{i=1}^m \mathbb{P}(\xi_i\le g_i(x))\ge p\right\}
$
is convex for all $p>\max_{1\le i\le m} F_{\zeta_i}(t_i^*)$.
\end{corollary}

\begin{proof}
By construction, we have $\mathbb{P}(\xi_i\le g_i(x))=\mathbb{P}(\zeta_i\le \tilde g_i(x))$ $(i=1,\dots,m)$. Thus, the assertions follow immediately from Theorem \ref{thm:eventual_convexity_standardized}.
\end{proof}

\begin{remark}
    Assuming $\mu=0$ and $\sigma=1$ in Theorem 6 incurs no loss of generality. For a general variable $\zeta \sim GH_1(\lambda, \chi, \psi, \mu, \sigma^2, \varphi)$, the affine transformation $\xi = (\zeta - \mu)/\sigma$ introduces terms that are asymptotically dominated by the linear term $\sigma s$ for large $t$. Specifically, applying the chain rule yields $t f_\xi'(t) + \alpha f_\xi(t) = \frac{1}{\sigma^2} \Bigl[ (\sigma s + \mu) f_\zeta'(s) + \alpha \sigma f_\zeta(s) \Bigr]$, where $s:=(t-\mu)/\sigma$.
    In fact, as $t \to \infty$ (which implies $s \to \infty$), the linear term $\sigma s$ asymptotically dominates the constant $\mu$. Therefore, for sufficiently large $t$, the sign of the entire expression is determined by $\sigma \bigl[s f_\zeta'(s) + \alpha f_\zeta(s)\bigr]$.
    Furthermore, compared to existing results focusing on symmetric elliptical distributions $(\gamma = 0)$ \cite{cheng2014second,henrion2008convexity,minoux2016convexity}, Theorem \ref{thm:eventual_convexity_standardized} extends eventual convexity to skewed distributions $(\gamma \neq 0)$. This provides a unified guarantee for a broader range of practical applications involving asymmetric uncertainties, such as the Variance Gamma or skewed Student's t models. Note that we rely on the separable structure here. More precisely, using Proposition \ref{1-dimensional generalized hyperbolic distributions} for $\xi \sim GH_{N}(\lambda, \chi, \psi, \mu, \Sigma, \gamma)$, the standardization of $\xi^{\top} x$ yields $\zeta(x) \sim GH_{1}(\lambda, \chi, \psi, 0, 1, \frac{x^{\top}\gamma}{\sqrt{x^{\top}\Sigma x}})$, which incorporates a decision-dependent skewness parameter $x^{\top}\gamma/\sqrt{x^{\top}\Sigma x}$.
\end{remark}



\begin{theorem}
\label{prop:convex_optimization_regime}
Under the assumptions of Theorem~\ref{thm:eventual_convexity_standardized}, define $\Psi(x):=\sum_{i=1}^m \log F_{\xi_i}(g_i(x))
$, $p^{*}:=\max_{1\le i\le m}F_{\xi_i}(t_i^{*})$ and $\mathcal D_p:=\{x\in X:\Psi(x)\ge \log p\}$. For every confidence level \(p\in(p^{*},1)\), the CCO problem \eqref{CCP} is equivalent to
\begin{equation}\label{CCP-rewritten}
    \min_{x\in X} c^\top x
    \quad \text{subject to} \quad
    \Psi(x)\ge \log p.
\end{equation}
Moreover, $\Psi(x)$ is concave on $\mathcal{D}_p$ and \eqref{CCP-rewritten} is a convex optimization problem. 
If \(g_i\) is continuously differentiable for each \(i=1,\dots,m\), then for every \(x\) such that
\(F_{\xi_i}(g_i(x))>0\) for all \(i\),
$\nabla \Psi(x)
=
\sum_{i=1}^m
\frac{f_{\xi_i}(g_i(x))}{F_{\xi_i}(g_i(x))}
\,\nabla g_i(x).
$
Moreover, if there exists a point \(x^{\circ}\in \operatorname{ri}(X)\) such that
\(\Psi(x^{\circ})>\log p\), then a point \(x^{*}\in X\) is optimal if and only if there exists a multiplier \(\lambda^{*}\ge 0\) such that
\[
\begin{cases}
0 \in c - \lambda^{*} \nabla \Psi(x^{*}) + N_X(x^{*}),\\[1mm]
\Psi(x^{*})\ge \log p,\\[1mm]
\lambda^{*} \bigl(\Psi(x^{*})-\log p\bigr)=0,
\end{cases}
\]
where \(N_X(x^{*})\) denotes the normal cone of \(X\) at \(x^{*}\), and
\(\operatorname{ri}(X)\) denotes the relative interior of \(X\). Furthermore, define $v(\tau):=\inf\{c^\top x:\ x\in X,\ \Psi(x)\ge \tau\}$. If \(v\) is differentiable at \(\tau\), then $v'(\tau)=\lambda^{*}(\tau)\ge 0$, where \(\lambda^{*}(\tau)\) is the optimal Lagrange multiplier associated with the constraint \(\Psi(x)\ge \tau\).
For the optimal value as a function of the probability level \(p\), define $\widetilde v(p):=v(\log p)$. Then, whenever \(v\) is differentiable at \(\tau=\log p\), $\frac{d\widetilde v}{dp}(p)=\frac{\lambda^{*}(\log p)}{p}\ge 0$.
\end{theorem}
\begin{proof}
    Since $p \in (p^{*}, 1)$, the product constraint is equivalent to $\Psi(x) \ge \log p$, yielding \eqref{CCP-rewritten}. Under the assumptions of Theorem \ref{thm:eventual_convexity_standardized}, the proof of Theorem 3.1 in \cite{henrion2008convexity} implies $\Psi(x)$ is concave on $\mathcal{D}_p$, making \eqref{CCP-rewritten} a convex optimization problem. The gradient $\nabla \Psi(x)$ follows directly from the chain rule. Finally, since \eqref{CCP-rewritten} is a standard convex program, the necessary and sufficient KKT conditions and the sensitivity result $\frac{d\tilde{v}}{dp}(p) = \frac{\lambda^{*}(\log p)}{p} \ge 0$ follow directly from classical duality and envelope theorems under Slater's condition.
\end{proof}

\begin{remark}
The previous result shows that the GH chance constraint is not only eventually convex for sufficiently large confidence levels, but also numerically tractable in the sense that its objective function and first-order information can be evaluated explicitly from the GH density and distribution function. In particular, the difficult part of the model is confined to the numerical evaluation of $F_{\xi_i}$ and the thresholds $t_i^{*}$, while the optimization step can be carried out by standard convex optimization machinery whenever $p>p^{*}$.
\end{remark}

\section{Numerical experiments}\label{section about numerical experiments}

This section presents numerical experiments to illustrate the two main theoretical contributions of this paper: the computation of the threshold $t^*(\alpha)$ associated with $\alpha$-decreasing densities, and the eventual convexity of the feasible set in the GH chance-constrained setting.

All numerical experiments were implemented in Python 3.12.8 using NumPy 2.2.4 and SciPy 1.15.2. The computations were executed on a Windows 11 machine equipped with an 8-core CPU, running on default threads, and 15.6 GB of RAM. For the optimization steps, we utilized the sequential least squares programming (SLSQP) solver. The stopping and feasibility tolerances were both set to $10^{-9}$, with a maximum of 1000 iterations. Analytical gradients were provided to the solver, and no warm starts were used.

We first consider a one-dimensional GH distribution $\xi \sim GH_1(\lambda,\chi,\psi,0,1,\varphi)$ with $(\lambda,\chi,\psi,\varphi)=(0.8,\,1.5,\,1.2,\,-0.6)$. Since $\psi>0$, the distribution falls under Case 1 of Theorem \ref{alpha-decreasing of density function}, and the threshold $t^*(\alpha)$ is well defined for every $\alpha\in\mathbb{R}$. To compute $t^*(\alpha)$, we evaluate the function
$D(t;\alpha):=t f_\xi'(t)+\alpha f_\xi(t)$, where $f_{\xi}$ is given by \eqref{non-integral form _ density of GH_d random vector with W density form} and $f_\xi'(t)$ is computed via central finite differences. 
We identify the final point where $D(\cdot;\alpha)$ switches sign, going from positive to negative.
Then, the threshold $t^*(\alpha)$ is refined by a Brent's method on the corresponding bracketing interval. Figure~\ref{fig:tstar-gh} shows: (a) the normalized density of $\xi$, (b) the functions $D(\cdot;\alpha)$ with different values of $\alpha$, and (c) the dependence of $t^*(\alpha)$ on $\alpha$. Table \ref{tab:t_star_computation} summarizes the exact thresholds $t^*(\alpha)$, bracketing intervals, and CPU times for all tested $\alpha$. The numerical results confirm the eventual $\alpha$-decreasing property. For $\alpha \le 0$, the function $D(\cdot; \alpha)$ is already negative on $[10^{-2}, 10^3]$, yielding thresholds below $10^{-2}$. For $\alpha > 0$, the exact thresholds are computed efficiently within milliseconds.

\begin{table}[htbp]
    \centering
    \caption{Computation results of $t^*(\alpha)$ for various values of $\alpha$. For $\alpha \le 0$, the condition is satisfied at the initial grid point $t=0.01$, requiring no bisection (denoted by '---').}
    \label{tab:t_star_computation}
    \resizebox{\textwidth}{!}{%
    \begin{tabular}{lcccccccc}
        \toprule
        \textbf{$\alpha$} & \textbf{$-3.0$} & \textbf{$-1.5$} & \textbf{$0.0$} & \textbf{$0.5$} & \textbf{$1.0$} & \textbf{$2.0$} & \textbf{$3.0$} & \textbf{$5.0$} \\
        \midrule
        $t^*(\alpha)$ & $<0.01$ & $<0.01$ & $<0.01$ & $0.4608$ & $0.7580$ & $1.2730$ & $1.7700$ & $2.7806$ \\
        Bracket & --- & --- & --- & $[0.461, 0.462]$ & $[0.758, 0.759]$ & $[1.272, 1.275]$ & $[1.768, 1.772]$ & $[2.776, 2.783]$ \\
        CPU Time (ms) & $214.7$ & $296.0$ & $167.1$ & $167.7$ & $190.5$ & $239.3$ & $196.4$ & $184.0$ \\
        \bottomrule
    \end{tabular}%
    }
\end{table}

\begin{figure*}[t]
\centering
\includegraphics[width=0.89\textwidth]{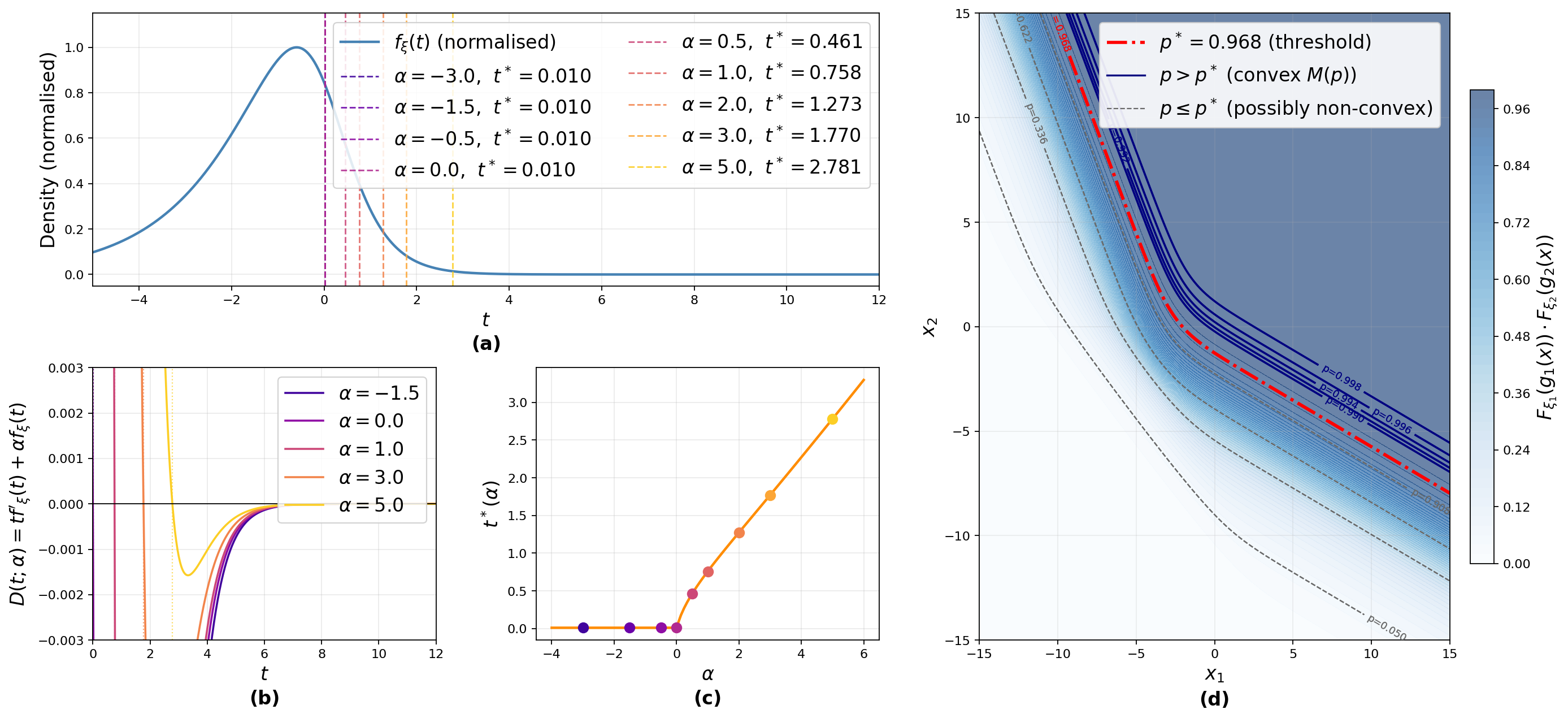}
\caption{(a) GH density; (b) $D(t;\alpha):=t f_\xi'(t)+\alpha f_\xi(t)$; (c) the threshold $t^*(\alpha)$ of the $\alpha$-decreasing property for GH distributions, where $\xi \sim GH_1(0.8,1.5,1.2,0,1,-0.6)$. (d) The critical level $p^*$ of the feasible set $M(p)$ for the two-dimensional GH chance constraint, where $\xi = (\xi_{1}, \xi_{2})^{\top}$ with $\xi_1 \sim GH_1(0.8,1.5,1.2,0,1,-0.6)$ and $\xi_2 \sim GH_1(1.2,1.0,0.8,0,1,-0.4)$.}
\label{fig:tstar-gh}
\end{figure*}

We next illustrate the eventual convexity result for the joint chance-constrained feasible set
$M(p)=\left\{x\in X:\prod_{i=1}^2 \mathbb{P}\bigl(\xi_i\le g_i(x)\bigr)\ge p\right\}$. 
We consider two independent GH random variables $\xi_1 \sim GH_1(0.8,1.5,1.2,0,1,-0.6)$, $\xi_2 \sim GH_1(1.2,1.0,0.8,0,1,-0.4)$, and let $r_1=r_2=1$, $g_1(x_1,x_2)=0.8x_1+0.4x_2+3.5$, $g_2(x_1,x_2)=0.4x_1+0.9x_2+3.0$. Since $g_1$ and $g_2$ are linear, they are $(-r_i)$-concave for any $r_i>0$. Thus, Theorem~\ref{thm:eventual_convexity_standardized} applies with $\alpha_i=r_i+1=2$. The corresponding thresholds are $t_1^*=t^*_{\,\xi_1}(2)\approx 1.2730$, $t_2^*=t^*_{\,\xi_2}(2)\approx 1.7008$. Using numerical integration of the GH densities, we obtain $F_{\xi_1}(t_1^*)\approx 0.9681$, $F_{\xi_2}(t_2^*)\approx 0.9626$, $p^*=\max\{F_{\xi_1}(t_1^*),F_{\xi_2}(t_2^*)\}\approx 0.9681$. Thus, the feasible set $M(p)$ is guaranteed to be convex for every $p>0.9681$. Figure~\ref{fig:tstar-gh} (d) displays the contour structure of $h(x)=F_{\xi_1}\bigl(g_1(x)\bigr)\,F_{\xi_2}\bigl(g_2(x)\bigr)$, which shows the eventual convexity property of the feasible sets.


From a practical perspective, computing the exact threshold $p^*$ provides significant computational advantages for solving optimization problems. CCO problems are generally non-convex and computationally intractable. However, by verifying $p \ge p^*$ a priori, practitioners are mathematically guaranteed that the feasible set $M(p)$ is convex. As a result, the associated CCO problem can be solved globally using standard convex optimization algorithms, without resorting to expensive global heuristics or facing the risk of convergence to local minima.

Based on the same GH specifications as the above example and the same critical level \(p^*=0.9681\), we next consider the nonlinear GH chance-constrained optimization problem
\begin{equation}\label{example-2}
    \min_{x\in X} \ c^\top x
    \quad \text{s.t.} \quad
    \prod_{i=1}^2 \mathbb{P}\bigl(\xi_i \le g_i(x)\bigr)\ge p,
\end{equation}
where \(c=(1,1)^\top\), \(X=[-2,2]^2\), and $g_1(x_1,x_2)=(x_1^2+x_2^2+0.1)^{-1}$, $g_2(x_1,x_2)=((x_1+x_2)^2+0.1)^{-1}$. Both \(g_1\) and \(g_2\) are \((-1)\)-concave on \(X\), so the eventual convexity result applies with the same \(\alpha_i=2\) as in the previous example. We test the confidence levels $
p\in\{0.9731,\ 0.9781,\ 0.9831,\ 0.9881,\ 0.9931\}$, all of which satisfy \(p>p^*\). The problem is solved numerically by the \textsc{SLSQP} method. 
\begin{figure*}[t]
\centering
\includegraphics[width=1.0\textwidth]{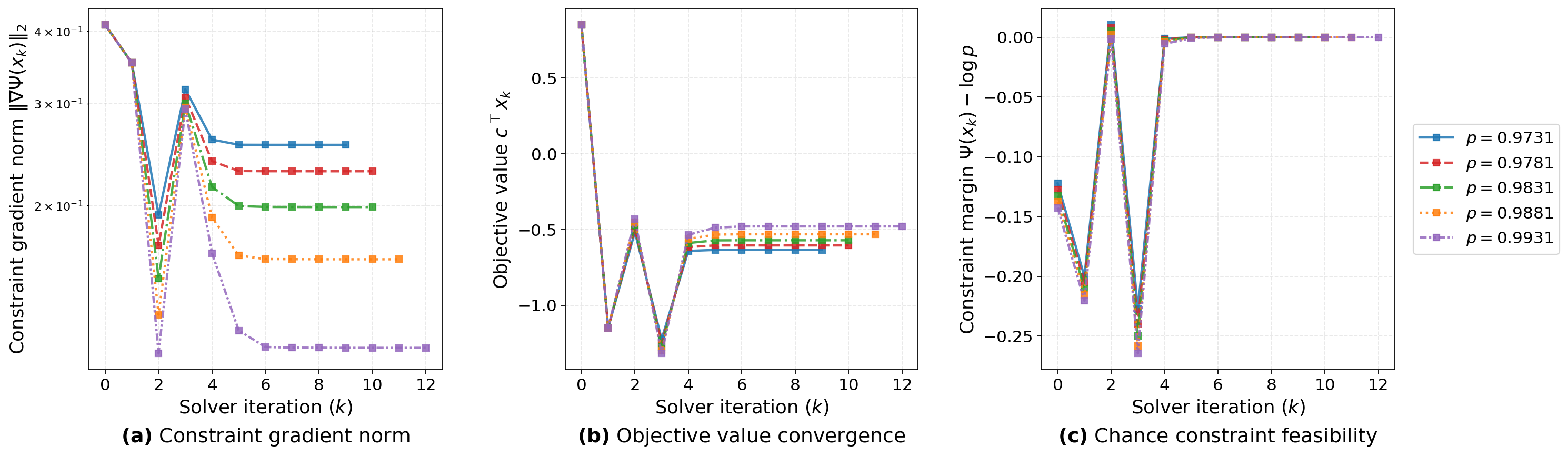}
\caption{Iteration-wise convergence results for the GH chance-constrained optimization problem \eqref{example-2} under different confidence levels \(p\), where \(\Psi(x):=\sum_{i=1}^m \log F_{\xi_i}(g_i(x))\).}
\label{fig:gradient_convergence}
\end{figure*}
Figure~\ref{fig:gradient_convergence} reports the iteration-wise behavior of the \textsc{SLSQP} solver under these confidence levels. In particular, the constraint-gradient norm \(\|\nabla \Psi(x_k)\|_2\), the objective value \(c^\top x_k\), and the feasibility margin \(\Psi(x_k)-\log p\), where
$\Psi(x):=\sum_{i=1}^m \log F_{\xi_i}(g_i(x))$, all stabilize after a small number of iterations, indicating fast convergence in the certified convex range. 
\begin{table}[htbp]
\centering
\caption{Solver results for the GH chance-constrained optimization problem \eqref{example-2} under different confidence levels \(p\). The critical threshold for eventual convexity is \(p^{*} = 0.9681\). All reported instances successfully converged to a feasible solution.}
\label{tab:optimization_results_nonlinear}
\begin{tabular}{cccccc}
\toprule
\(p\) & \(p>p^{*}\) & \(x^{*}=(x_1^{*},x_2^{*})^\top\) & Obj.\ value & Iterations & CPU time (ms) \\
\midrule
0.9731 & \checkmark & \((-0.3172, -0.3172)\) & \(-0.63447\) & 9  & 78.18  \\
0.9781 & \checkmark & \((-0.3023, -0.3023)\) & \(-0.60452\) & 10 & 77.26  \\
0.9831 & \checkmark & \((-0.2854, -0.2854)\) & \(-0.57085\) & 10 & 82.89  \\
0.9881 & \checkmark & \((-0.2655, -0.2655)\) & \(-0.53104\) & 11 & 138.02 \\
0.9931 & \checkmark & \((-0.2395, -0.2395)\) & \(-0.47906\) & 12 & 107.04 \\
\bottomrule
\end{tabular}
\end{table}


\begin{table}[ht]
  \centering
  \caption{Multi-start stability and computational time tracking for the GH chance-constrained optimization problem ($p=0.9881$, \(p^{*} = 0.9681\)). The solver is initialized from five different starting points $x_0$.}
  \label{tab:multi_start_results}
  \resizebox{\textwidth}{!}{%
  \begin{tabular}{ccccc}
    \toprule
    Initial point $x_0$ & Final solution $x^{*}$ & Obj. value & Iterations & CPU time (ms) \\
    \midrule
    $(1.3399, 0.3862)$ & $(-0.2655, -0.2655)$ & -0.53104 & 12 & 121.36 \\
    $(1.7741, -0.5623)$ & $(-0.2655, -0.2655)$ & -0.53104 & 14 & 114.76  \\
    $(-1.3600, 0.3438)$ & $(-0.2655, -0.2655)$ & -0.53104 & 12 & 92.34 \\
    $(-0.7493, -1.1027)$ & $(-0.2655, -0.2655)$ & -0.53104 & 11 & 89.85 \\
    $(1.3543, 0.7686)$ & $(-0.2655, -0.2655)$ & -0.53104 & 11 & 94.94 \\
    \midrule
    \multicolumn{2}{l}{\textbf{Objective value mean $\pm$ std:}} & \multicolumn{3}{l}{$-0.53104 \pm 0.00000$} \\
    \multicolumn{2}{l}{\textbf{CPU time mean $\pm$ std:}} & \multicolumn{3}{l}{$102.65 \pm 14.37$ ms} \\
    \bottomrule
  \end{tabular}%
  }
\end{table}

Table \ref{tab:optimization_results_nonlinear} records the corresponding solution, objective value, iteration count, and CPU time. The iteration count remains small, between 9 and 12, and the CPU time stays within the millisecond scale. To provide further supporting evidence and account for the millisecond-scale execution times, we performed multi-start runs from five different initial points for the case $p=0.9881$. As reported in Table \ref{tab:multi_start_results}, all five runs converged to the same objective value ($-0.53104 \pm 0.00000$). The computational times across these repeated runs yielded a mean of $102.65$ ms with a standard deviation of $14.37$ ms. This two-dimensional example serves as a clear illustration of the computational tractability.

\section{Conclusion}
In this paper, we prove the convexity of chance constraints with skewed distributions. Specifically, relying on the framework of separable chance constraints, we first prove that the GH density is $\alpha$-decreasing. Then, we establish the eventual convexity with some skewed distributions. Finally, we provide numerical experiments to show the practical computability of the $\alpha$-decreasing property threshold $t^*(\alpha)$, the eventual convexity of the feasible sets, and the tractability of solving the optimization problems in the GH setting.



\section*{Acknowledgements}

This research was supported by French government under the France 2030 program, reference ANR-11-IDEX-0003 within the OI H-Code. The second co-author was supported by the French National Research Agency ANR PRCI Project CyberRO (Ref: ANR-25-CE10-5230)/RGC Joint Research Scheme (Ref: A-HKU702/25).











\bibliographystyle{plain}

\bibliography{reference}



\end{document}